\documentclass[12pt]{amsart}
\usepackage{amssymb,amsmath,amscd}
\usepackage{latexsym,bm,bbm,mathrsfs}
\usepackage{hyperref,graphicx}
\usepackage{mathtools}
\usepackage{stmaryrd}
\usepackage[all,pdf]{xy}
\usepackage{graphicx}
\usepackage{color}
\graphicspath{ {./} }

\newtheorem{thm}{Theorem}[section]

\newtheorem{prop}[thm]{Proposition}
\newtheorem{lemma}[thm]{Lemma}

\newtheorem{lem}[thm]{Lemma}
\newtheorem{defn}[thm]{Definition}

\newcommand{\bbc}{\mathbb{C}}

\newcommand{\bbq}{{\mathbb{Q}}}
\newcommand{\bbr}{{\mathbb{R}}}
\newcommand{\bbz}{{\mathbb{Z}}}
\newcommand{\cO}{{\mathcal{O}}}

\newcommand{\Aut}{\operatorname{Aut}}
\newcommand{\Lie}{\operatorname{Lie}}
\newcommand{\Res}{\operatorname{Res}}
\newcommand{\Hilb}{\operatorname{Hilb}}

\newcommand{\Sp}{\operatorname{Sp}}
\newcommand{\Jac}{\operatorname{Jac}}

\newcommand{\Hom}{\operatorname{Hom}}
\newcommand{\End}{\operatorname{End}}

\renewcommand{\Im}{\operatorname{Im}}

\title{Shimura Varieties, Kummer Varieties, and Rational Curves}

\author{Bo-Hae Im}
\address{Department of Mathematical Sciences, KAIST, 291 Daehak-ro, Yuseong-gu, Daejeon, 34141, South Korea}
\email{bhim@kaist.ac.kr}
\thanks{Bo-Hae Im was supported by Basic Science Research Program through the National Research Foundation of Korea(NRF) funded by the Ministry of  Science \& ICT(NRF-2020R1A2B5B01001835).}

\author{Michael Larsen}
\address{Department of Mathematics, Indiana University, Bloomington, IN, 47405, U.S.A.}
\email{mjlarsen@indiana.edu}
\thanks{Michael Larsen was partially supported by NSF grant DMS-1702152.}

\date{\today}

\author{Sailun Zhan}
\address{Department of Mathematics, Indiana University, Bloomington, IN, 47405, U.S.A.}
\email{zhans@iu.edu}

\begin{document}

\maketitle

\begin{abstract}
For a very general product $A$ of seven or more elliptic curves, every rational curve on the Kummer variety of $A$ projects trivially onto the Kummer variety of at least
one of its factors.  As a consequence, a very general member of certain families of abelian varieties parametrized by connected Shimura  varieties of unitary type
has the property that its Kummer variety has no rational curves.
\end{abstract}

\section{Introduction}
Let $A/\bbc$ be an abelian variety and $G\subset \Aut(A)$ a finite group of automorphisms of $A$
as an abelian variety.  We would like to understand the rational curves on the quotient $A/G$.
At one end of the spectrum, $A/G$ can be a rational variety.  A celebrated result of Looijenga \cite{L76} asserts that
$E^n/W$ is a weighted projective space when $E$ is an elliptic curve
and $W$ is the Weyl group of a root system of rank $n$
acting on $E^n$ through its action on the root lattice.  Koll\'ar and Larsen \cite{KL09} showed that
in general $A/G$ is uniruled if and only if $G$ fails the \emph{Reid-Tai condition}.  This means that there exists an automorphism $g\in G$ whose eigenvalues on $\Lie(A)$ are $e^{2\pi i x_1},\ldots,e^{2\pi i x_n}$ where
$0\le x_1\le x_2\le\cdots\le x_n<1$, and $0<\sum x_i < 1$.  Im and Larsen \cite{IL15} showed that if
$\sum x_i = 1$, $A/G$ must still have at least one rational curve.

At the other end of the spectrum, we have the classical result that $A$ itself contains no rational curves.
Pirola \cite{Pi89} proved that for a very general abelian variety of dimension $n\ge 3$, the Kummer variety
$A/\{\pm 1\}$ has no rational curves.  Unfortunately, we do not know much about the locus for which Kummer varieties of dimension $n\ge 3$ admit rational curves.  In this paper we show that there are no rational curves on Kummer varieties associated to a very general point of certain Shimura subvarieties of Siegel moduli spaces.

Our strategy is to find a  set of abelian varieties which are isogenous to products of elliptic curves and which are dense in the desired Shimura variety.  Of course, Kummer varieties of
products of elliptic curves always have some rational curves but they may not have enough to permit there to be a rational curve for a sufficiently general point in the closure.  We introduce the following definition.

\begin{defn}
Let $A:=E_1\times \cdots \times E_n$ be a product of elliptic curves and let $C$ be a curve on $\bar A := A/\{ \pm 1\}$, where the involution maps $P$ to $-P$ for $P\in A$. The curve $C$ is called {{\emph{diagonal}}} if the projection $C\to\bar{E_{i}}$ is nonconstant for each $i$.
\end{defn}

Throughout this paper, a variety will always be an integral separated scheme of finite type over $\bbc$.

\section{Products of elliptic curves}

For $i=1,2,\ldots, n$, let $E_i$ be an elliptic curve whose affine model is given by $y_i^2 = x_i(x_i-1)(x_i-\lambda_i)$ for some $\lambda_i\in \bbc\setminus\{0,1\}$. The main result in this section is the following:
\begin{thm}\label{diagonal}
Let $n\geq 7$ be an integer.  For a very general $(\lambda_{1},\ldots,\lambda_{n})$ in ${{{\bbc}^{n}}}$, there is no diagonal rational curve on the Kummer variety
$\overline{E_1\times\cdots\times E_n}$.
\end{thm}

Before we give the proof of this theorem, let us make the following observation. If $C$ is a diagonal curve on the Kummer variety, $\overline{E_1\times\cdots\times E_n}$, let $t$ denote a generator of the function field of $C$. As $C$ is diagonal, for each $i$, the non-constant morphism $\psi_i\colon C\to \bar E_i$ determines a non-constant rational function $R_i(t) = P_i(i)/Q_i(t) = \psi_i^*(x_i)$ in the function field of $C$ for some polynomials $P_i, Q_i$.

Let
\[
S_i(t) = P_i(t) (P_i(t)-Q_i(t))(P_i(t)-\lambda_i Q_i(t))Q_i(t).
\]
As
$${{S_{i}(t)}} = Q_i(t)^4 \psi_i^*(x_i(x_i-1)(x_i-\lambda_i)) = (Q_i(t)^2 \psi_i^*(y_i))^2,$$
we have for all $i$ and $j$,
$$\sqrt{S_iS_j} = Q_i^2 Q_j^2 \psi_i^*(y_i) \psi_j^*(y_j)$$
in the function field of $C$.

We claim that if $n\geq 7$, there must be at least three coordinates $i,j,k$ for which $S_{i}$, $S_{j}$, and $S_{k}$ are the same in $\bbc(t)^{\times}/\bbc(t)^{\times 2}$. Indeed, otherwise, there would have to be at least $4$ different cosets of $\bbc(t)^{\times 2}$ determined by the $S_{i}$'s. Then the function field of $C$ would contain the field
\[
\bbc(t,\sqrt{S_{1}S_{2}},\sqrt{S_{1}S_{3}}, \ldots, \sqrt{S_{1}S_{n}}),
\]
which is a $(\bbz/2\bbz)^m$-extension of $\bbc(t)$, where $m\geq 3$.  This is impossible by the following lemma.

\begin{lemma}
For $m\ge 3$, any $(\bbz/2\bbz)^m$-extension of $\bbc(t)$ has genus $\ge 1$.
\end{lemma}

\begin{proof}
Denote its genus by $g$ and the number of the branch points of the extension by $r$. Then by the Riemann-Hurwitz formula, we have
\[
2g-2=2^{m}(-2)+2^{m-1}r.
\]
The right hand side of the equation is divisible by $4$, so $g$ must be odd.
\end{proof}

We assume henceforth that
$$S_{1}\equiv S_{2}\equiv S_{3}\pmod{\bbc(t)^{\times 2}}.$$

We have the following lemma.

\begin{lemma}\label{kummer}
There are countably many rational curves on $\overline{E_{1}\times E_{2}}$.
\end{lemma}

\begin{proof}
Let $\pi$ denote the quotient morphism $\overline{E_1\times E_2}\to \bar E_1\times \bar E_2$.
Thus $\pi^{-1}$ applied to the affine $(x_1,x_2)$-plane in $\bar E_1\times \bar E_2$
has coordinate ring
$$\bbc[x_1,x_2,y_1y_2] = \bbc[x_1,x_2,z]/(z^2 - x_1(x_1-1)(x_1-\lambda_1)x_2(x_2-1)(x_2-\lambda_2)).$$
Let $C$ be a rational curve on $\overline{E_1\times E_2}$, so the projection of $C$ onto $\bar E_i$
is given by $x_i\mapsto P_i(t)/Q_i(t)$ for $i=1,2$.
Thus $\pi^{-1}(\pi(C))$ has two irreducible components if and only if
$$\frac{P_1}{Q_1}\left(\frac{P_1}{Q_1}-1\right)\left(\frac{P_1}{Q_1}-\lambda_1\right)
\frac{P_2}{Q_2}\left(\frac{P_2}{Q_2}-1\right)\left(\frac{P_2}{Q_2}-\lambda_2\right)$$
is a square in $\bbc(t)$, i.e., if and only if $S_1 S_2$ is a square in $\bbc(t)$.
Otherwise, it has one irreducible component, and the fact that the curve $z^2-S_1S_2=0$
has genus $0$ is equivalent to the statement that there are exactly two points at which
$S_1 S_2$ has a zero or pole of odd order.

In other words, if $\pi^{-1}(\pi(C))$ has  two irreducible components then there exist distinct  $c_{1},\ldots, c_{k}\in\bbc$ such that
$$S_{1}\equiv S_{2}\equiv(t-c_{1})\ldots(t-c_{k})\pmod{\bbc(t)^{\times 2}}.$$
If it has one, then there exist distinct  $c_{1},\ldots, c_{k+1}\in\bbc$ such that
\begin{align*}
S_{1}&\equiv(t-c_{1})\ldots(t-c_{k-1})(t-c_{k}) \pmod{\bbc(t)^{\times 2}},\\
S_{2}&\equiv(t-c_{1})\ldots(t-c_{k-1})(t-c_{k+1})
\pmod{\bbc(t)^{\times 2}}.
\end{align*}
We deal with the one-component case; the two-components case works in the same way.

Fix the degrees of polynomials $P_{1}, Q_{1}, P_{2}$ and $Q_{2}$, where the pairs $P_1,Q_1$ and $P_2,Q_2$ are relatively prime.
We assume the $P_i$ and $Q_i$ are normalized so that the resultants $\Res(P_1,Q_1)$ and $\Res(P_2,Q_2)$ equal $1$.
Then there are only finitely many values $k$ such that the system of equations
\begin{equation}
\label{two equations}
\begin{split}
\begin{cases} P_{1}(t) (P_{1}(t)-Q_{1}(t))(P_{1}(t)-\lambda_{1} Q_{1}(t))Q_{1}(t)&=(t-c_{1})\ldots(t-c_{k})A(t)^{2},\\
P_{2}(t) (P_{2}(t)-Q_{2}(t))(P_{2}(t)-\lambda_{2} Q_{2}(t))Q_{2}(t)&=(t-c_{1})\ldots(t-c_{k})B(t)^{2},\\
\Res(P_1,Q_1)&=1,\\
\Res(P_2,Q_2)&=1,\\
\end{cases}
\end{split}
\end{equation}
in unknown polynomials $P_1,Q_1,P_2,Q_2,A,B\in\bbc[t]$ and unknown constants $c_1,\ldots,c_k\in\bbc$ could possibly have a solution.

The system \eqref{two equations} defines a closed subvariety in an affine space over $\bbc$ parametrizing the coefficients of $P_1,Q_1,P_2,Q_2,A,B$ and the constants $c_i$. Each closed point corresponds to a rational curve on the Kummer surface. But since there is no 1-dimensional family of rational curves on a K3 surface \cite[Chapter 13, 0.1]{H16}, this closed subvariety must be 0-dimensional. In other words, there are only finitely many rational curves if we fix the degree of $P_{1}, Q_{1}, P_{2}$ and $Q_{2}$. Since each rational curve on the Kummer surface comes from a choice of $P_{1}, Q_{1}, P_{2}$ and $Q_{2}$, there are only countably many of them.
\end{proof}

Now we fix $\lambda_{1}$ and $\lambda_{2}$, and suppose we have $S_{1}\equiv S_{2}\equiv(t-c_{1})\ldots(t-c_{k}) $. From Lemma~\ref{kummer} in the above, we know there are only countably many possibilities for $S_i\pmod{\bbc(t)^{\times 2}}$ arising from rational curves on $\overline{E_1\times E_2}$.

\begin{lemma}\label{countable}
For any fixed rational function $D\in \bbc(t)^\times$, there exist only countably many $\lambda\in\bbc$ for which there exist non-constant rational functions $\frac PQ,T\in\bbc(t)$ such that
\begin{equation}
\label{star}
\frac{P}{Q}\left(\frac{P}{Q}-1\right)\left(\frac{P}{Q}-\lambda\right)=DT^{2}.
\end{equation}
\end{lemma}

\begin{proof}
Suppose that $\lambda\not\in\{0,1\}$ and polynomials $P,Q$ and a rational function $T$ exist satisfying \eqref{star}.  Let $E(\lambda)$ denote the elliptic curve whose affine model is given by $y^{2}=x(x-1)(x-\lambda)$ and define
\[
X\colon u^{2}=\frac{P(v)}{Q(v)}\left(\frac{P(v)}{Q(v)}-1\right)\left(\frac{P(v)}{Q(v)}-\lambda\right).
\]
The key observation here is that the function field of $X$ is independent of the choice of $P,Q,T$. Hence, the complete non-singular curve $Y$ with the same function field as $X$
depends only on $\lambda$ and on the coset $D\bbc(t)^{\times 2}$.

Indeed, it is clear that $u_1^2 = T_1^2D(v_1)$ is birationally equivalent to $u_2^2 = T_2^2 D(v_2)$, so if
$(\lambda_{1},P_{1},Q_{1},T_{1})$ and $(\lambda_{2},P_{2},Q_{2},T_{2})$ both satisfy \eqref{star}, then
\[
X_1\colon u_{1}^{2}=\frac{P_{1}(v_{1})}{Q_1(v_{1})}\left(\frac{P_{1}(v_{1})}{Q_{1}(v_{1})}-1\right)\left(\frac{P_{1}(v_{1})}{Q_{1}(v_{1})}-\lambda\right)
\]
and
\[
X_2\colon u_{2}^{2}=\frac{P_{2}(v_{2})}{Q_2(v_{2})}\left(\frac{P_{2}(v_{2})}{Q_{2}(v_{2})}-1\right)\left(\frac{P_{2}(v_{2})}{Q_{2}(v_{2})}-\lambda\right)
\]
are birationally equivalent.

Now since $(v,u)\mapsto(P(v)/Q(v),u)$ defines a non-constant map $X\to E(\lambda)$, we also have a non-constant morphism $Y\to E(\lambda)$. By the universal property of the Jacobian variety, this map factors through $\Jac(Y)$. Hence there exists a surjective homomorphism $\Jac(Y)\rightarrow E(\lambda)$. For the abelian variety $\Jac(Y)$, there exists a product of simple abelian varieties $A_{1}\times\ldots\times A_{r}$ which is isogeneous to $\Jac(Y)$. So there is a surjective homomorphism $A_{1}\times\ldots\times A_{r}\rightarrow E(\lambda)$. Since this homomorphism is non-constant, there must be at least one factor $A_{i}$ such that when we restrict this map to $A_{i}$, we get a non-constant homomorphism. But since $A_{i}$ is simple, $A_{i}$ will also be an elliptic curve. Hence $E(\lambda)$ is isogeneous to $A_{i}$.

What we have just proved in the above is that if we fix the mod $2$ divisor $D$, then it determines a unique smooth curve $Y$. Suppose $\Jac(Y)$ is isogeneous to $A_{1}\times\ldots\times A_{r}$. If $(\lambda,P,Q,T)$ satisfies \eqref{star}, then $E(\lambda)$ must be isogeneous to one of the elliptic curve factors in $A_{1}\times\ldots\times A_{r}$. But since we know that there are only countably many elliptic curves isogenous to a given elliptic curve, we deduce that there are only countably many $\lambda$ corresponding to the given $D$.
\end{proof}

By Lemma \ref{countable} and the remark after Lemma \ref{kummer}, if we fix $\lambda_{1}$ and $\lambda_{2}$, then there are only countably many $\lambda_{3}$ such that $S_{1}\equiv S_{2}\equiv S_{3}$ for some $P_{1},Q_{1},P_{2},Q_{2}$ and $P_{3},Q_{3}$.

Now we can prove the following lemma.

\begin{lemma} \label{3E}
For a very general $(\lambda_1, \lambda_2, \lambda_3)$ in ${{{\bbc}^{3}}}$, there does not exist a diagonal rational curve in the Kummer variety $\overline{E_{1}\times E_{2}\times E_{3}}$ such that $S_{1}\equiv S_{2}\equiv S_{3}$ in ${{\bbc(t)^{\times}/\bbc(t)^{\times 2}}}$.
\end{lemma}

\begin{proof}
Consider a diagonal rational curve in the Kummer variety $\overline{E_{1}\times E_{2}\times E_{3}}$ such that $S_{1}\equiv S_{2}\equiv S_{3}$. Then we have
\[
P_i(t) (P_i(t)-Q_i(t))(P_i(t)-\lambda_i Q_i(t))Q_i(t)=D(t)T_{i}^{2}(t)
\]
for $i=1,2,3$.

Fix the degree of $P_{i}(t),Q_{i}(t),D(t)$, and regard the coefficients of $P_{i},Q_{i},T_{i},D$ and $\lambda_{i}$ as variables. These equations define a closed subvariety $Z$ in an affine space over ${\bbc}$. Consider the projection map $\pi$ from the affine space to ${{{\bbc}^{3}}}$, which corresponds to the three coordinates $\lambda_{1},\lambda_{2},\lambda_{3}$. Then the image of $Z$ under $\pi$ will be a constuctible set, which is a finite union of locally closed subsets. We want to claim that each locally closed subset has dimension $\leq 2$. Then this will imply that all the tuples $(\lambda_{1},\lambda_{2},\lambda_{3})$ which admits the above diagnal rational curve are contained in countably many proper closed subset of ${{{\bbc}^{3}}}$.

We prove the claim by contradiction. Suppose there is a locally closed subset $W$ of dimension {$3$}. Then it is an open dense subset of ${{{\bbc}^{3}}}$. Consider the projection map $p:{{{\bbc}^{3}}}\rightarrow{{{\bbc}^{2}}}$ to the first two coordinates. Then there must exist a fiber which has a non-empty intersection with $W$. But this intersection contains uncountably many points, which contradicts the remark after Lemma {{\ref{countable}}}.
\end{proof}

Now we are ready to prove Theorem \ref{diagonal}.

\begin{proof}[The proof of Theorem~\ref{diagonal}.]
If $n\geq 7$ and there is a diagonal rational curve on the Kummer variety, then by the {{observation after Theorem \ref{diagonal}}}, we know that there must be three coordinates $i,j,k$ for which $S_{i},S_{j}$ and $S_{k}$ are the same in ${{\bbc(t)^{\times}/\bbc(t)^{\times 2}}}$. That implies $(\lambda_{i},\lambda_{j},\lambda_{k})$ is contained in countably many proper closed subsets of ${{{\bbc}^{3}}}$ by Lemma ${{\ref{3E}}}$. Hence this $(\lambda_{1},\ldots,\lambda_{n})$ is contained in countably many proper closed subsets of ${{{\bbc}^{n}}}$. Since there are only finitely many way to choose three indices from $\{1,\dots,n\}$, any tuple $(\lambda_{1},\ldots,\lambda_{n})$ which admits a diagonal rational curve is also contained in countably many proper closed subset in ${{{\bbc}^{n}}}$. Hence for a very general $(\lambda_{1},\ldots,\lambda_{n})$ in ${{{\bbc}^{n}}}$, there is no diagonal rational curve on the corresponding Kummer variety.
\end{proof}

\section{Families of Abelian Varieties}

Let $S$ denote a variety over $\bbc$ and $\pi\colon A\to S$ an abelian scheme over $S$ of relative dimension $g$.  Let $X$ be a reduced closed subscheme of $A$ which is symmetric in the sense that $X = -X$.

\begin{lem}
\label{constructibility}
The condition on $s\in S$ that the fiber $X_s$ generates the abelian variety $A_s$ is constructible.
\end{lem}

\begin{proof}
By \cite[I~Proposition 2.2]{Borel}, $X_s$ generates $A_s$ if and only if the multiplication morphism
$$\mu_s\colon \underbrace{X_s\times \cdots\times X_s}_{2g}\to A_s$$
is surjective.
If
$$\mu\colon X^{2g} := \underbrace{X\times_S \cdots\times_S X}_{2g}\to A$$
is the fiberwise multiplication morphism, then the set of $s\in S$ such that the fiber $X_s$ generates the abelian variety $A_s$
is the complement of $\pi(A\setminus \mu(X^{2g}))$, and is therefore constructible by Chevalley's theorem.
\end{proof}

\begin{defn}
A curve $C$ on the Kummer variety $\bar A$ associated to an abelian variety $A$ is  \emph{nondegenerate} if the hyperelliptic preimage of $C$ in $A$ generates the abelian variety $A$. \end{defn}

\begin{lem}
\label{diagonal vs nondegenerate}
If $E_1,\ldots,E_n$ are pairwise non-isogenous elliptic curves over $\bbc$, a rational curve on the Kummer variety of $A=E_1\times \cdots\times E_n$ is diagonal if and only if it is nondegenerate.
\end{lem}

\begin{proof}
Let $X$ denote the inverse image in $A$ of a rational curve $C$ in $\bar A$.  As $C$ is rational and $A$ does not contain any rational curve,
$X$ must consist of a single irreducible (hyperelliptic) component. 
If $C$ is not diagonal, then the image of $C$ in some $\bar E_i$ is a single point, so the image of $X$ in $E_i$ is finite (in fact, a single point).
implying that $X$ does not generate $A$.  

Conversely, if $C$ is degenerate, then $X$ does not generate $A$.  Fixing a translate $Y$ of $X$ which contains $0$, also $Y$ does not generate $A$.
Let $B$  denote the abelian subvariety of $A$ generated by $Y$.  Let $m$ be the smallest integer such that the projection $\pi$ from $B$ to the product of some $m$ factors $E_i$ of $A$
fails to be surjective.  Without loss of generality, we may assume that these factors are $E_1,\ldots,E_m$.  Suppose $m\ge 2$.
By Goursat's Lemma,  there exists an isomorphism between a non-trivial quotient of $E_1\times \cdots\times E_{m-1}$ and a non-trivial quotient $E'_m$ of $E_m$, which is necessarily an elliptic curve isogenous to $E_m$.  This gives a surjective morphism $E_1\times \cdots\times E_{m-1}\to E'_m$, whose restriction to some factor $E_i$ must be non-trivial.  This contradicts the fact that $E_i$ and $E_m$ are not isogenous.  Thus $m=1$, implying that the morphism $B\to E_1$ is trivial, which implies $X\to E_1$ is constant, so $C$ is non-diagonal.

\end{proof}

\begin{lem}
\label{surjection}
If $\pi\colon A\to B$ is a surjective morphism of abelian varieties and $\bar A$ contains a non-degenerate rational curve $C$, then $\bar B$ contains a non-degenerate rational curve as well.
\end{lem}

\begin{proof}
The morphism $\pi$ determines a morphism $\bar \pi \colon \bar A\to \bar B$.  Let $X\subset \bar A$ denote the inverse image of $C$ and $Y := \pi(X)$.  As $\pi$ is proper, $Y$ is closed; it is also connected and of dimension $\le 1$, so it is either a point or a curve.
As $X$ generates $A$, $Y$ generates $B$, so $Y$ is a curve.  The image of $Y$ in $\bar B$ is $\bar \pi(C)$.  Any $1$-dimensional image of a rational curve is again a rational curve.
\end{proof}

\begin{lem}
\label{genus 0}
Let $S$ be a variety over $\bbc$ and $\pi\colon C\to S$ a projective morphism.  The set of points $s\in S$ such that the fiber $C_s$ is a curve of genus $0$ (in particular, geometrically integral) is
constructible.
\end{lem}

\begin{proof}
By \cite[Th\'eor\`eme~11.1.1]{G66} and Noetherian induction, there exists a finite stratification of $S$ such that $\pi$ is flat over every stratum.  Therefore, without loss of generality, we may assume it is flat.  By \cite[Proposition~9.9.4(iv)]{G66} the subset of $s$ for which $C_s$ is geometrically reduced is constructible, so we may pass to the strata of a finite stratification and assume that
$\pi$ has geometrically reduced fibers.  By \cite[Corollaire~9.7.9]{G66}, the subset of $s$ for which $C_s$ is geometrically irreducible is constructible, so again we may assume this is true for every fiber of $\pi$.  Finally, for flat families of projective curves, geometric genus is lower semicontinuous \cite[Proposition~2.4]{DH88}, so the locus of genus $0$ curves is closed.
\end{proof}

\begin{prop}\label{family}
Suppose we have an abelian scheme $A$ of dimension $g$ over an irreducible base variety $S/\bbc$. Let $\bar A\to S$ be the associated family of Kummer varieties. Then one of the following must hold:
\begin{enumerate}
\item There exists a non-empty Zariski open subset $U\subset S$ such that for every $s\in U$, $\bar A_s$ has a nondegenerate rational curve.
\item  For a very general Kummer variety $\bar A_s$, there is no nondegenerate rational curve.
\end{enumerate}
\end{prop}
\begin{proof}
It suffices to divide the rational curves $C_V$ on fibers of $S$ into countably many families indexed by  schemes $V$ of finite type :
$$\xymatrix{
X_V\ar@{^{(}->}[r]\ar[dd]&A_V\ar[dr]\ar[dd]&\\
&&A\ar[dd]\\
C_V\ar@{^{(}->}[r]\ar[ddr]&\bar A_V\ar[dr]\ar[dd]&\\
&&\bar A\ar[dd]\\
&V\ar^(0.4){\pi_V}[dr]&\\
&&S}$$
By Lemma~\ref{constructibility}, the set $Z_V$ of $v\in V$ such that $(C_V)_v$ is nondegenerate
in $\bar A_{\pi_V(v)}$ (i.e., $X_{\pi_V(v)}$ generates $A_{\pi_V(v)}$) is constructible, so $\pi_V(Z_V)$ is a constructible subset of $S$.
If $\pi_V(Z_V)$ contains the generic point $\eta$ of $S$, then it contains a dense open subset of $S$,
so condition (1) holds.
Otherwise, for each $s$ in the complement of the countable union $\bigcup_V \overline{\pi_V(Z_V)}$,
the Kummer variety $\bar A_s$ has no nondegenerate rational curve, which implies condition (2).

To construct these families of rational curves,
we observe that every nondegenerate rational curve $C$ in a Kummer variety $\bar A_s$, corresponds to a point $[C]$ in a relative Hilbert scheme $\Hilb_{P}(\bar A/S)$ for some Hilbert polynomial $P$. Denote the structure map by $\pi\colon \Hilb_{P}(\bar A/S)\to S$. By Lemma~\ref{genus 0}, $[C]$ is actually contained ${{in}}$ some subvariety $V\subset \Hilb_{P}(\bar A/S)$ which is a stratum of $\Hilb_{P}(\bar A/S)$
parametrizing irreducible and reduced rational curves in $\bar A$.  We denote by $\pi_V$ the restriction of $\pi$ to $V$.

Define $\bar A_V := \bar A\times_S V$ and $A_V := A\times_S V$.  The quotient $S$-morphism $A\to \bar A$ defines a quotient $V$-morphism $A_V\to \bar A_V$.
By definition of the Hilbert scheme, $\bar A_V$ contains a closed subscheme $C_V$ which is flat over $V$ and such that each fiber $(C_V)_v\subset (\bar A_V)_v = \bar A_{\pi(v)}$ is a rational curve with Hilbert polynomial $P$.  Let $X_V:= C_V\times_{\bar A_V} A_V$, so the fiber of $X_V$ over $v$ is a hyperelliptic curve in $A_{\pi(v)}$.
\end{proof}

\section{Some Families of Unitary Type}

A general reference for the construction in this section is \cite[Chapter XXIV]{Shimura}.

Let $K$ be an imaginary quadratic field with ring of integers $\cO = \bbz + \bbz \tau$, where $\Im \tau > 0$.  Let $p\ge q\ge 7$ integers, and let $\Lambda_1= \bbz+\bbz \tau_1,\ldots,\Lambda_q = \bbz+\bbz \tau_q$, $\Im \tau_i>0$, denote lattices in $\bbc$.  Let $E$ (resp. $E_i$) denote the elliptic curve over $\bbc$ whose analytic space is isomorphic to $\bbc/\cO$ (resp. $\bbc/\Lambda_i$).
By Theorem~\ref{diagonal}, we may choose the $\Lambda_i$ such that the Kummer variety of $E_1\times \cdots\times E_q$ has no diagonal rational curves.
Since the property that two elliptic curves are not isogenous is very general on the moduli of the curves, we may further assume that $E_1,\ldots,E_q$ are mutually non-isogenous.
Define 
$$\Lambda := \cO^{p-q}\oplus \bigoplus_{i=1}^q \Hom(\cO,\Lambda_i),$$
and regard $\Lambda\otimes \bbr$ as a $\bbc$-vector space $V$ via the given inclusions $\Lambda_i\hookrightarrow \bbc$ and a choice of embeddings $K\hookrightarrow \bbc$.
Regarding $\Lambda$ as a free $\cO$-module of rank $p+q$ via the obvious action of $\cO$ on $\Hom(\cO,\Lambda)$, and extending by $\bbr$-linearity to define an action of $\cO$ on $V$, we obtain  a 
natural homomorphism $\iota$ from $\cO$ to the 
endomorphism ring of the complex torus $V/\Lambda$.   It also defines a second complex structure on $V$, which we call the $\iota$-structure.

We define a Hermitian form $H$ on $V$ as follows.  If $v = (a_1,\ldots,a_{p-q},\alpha_1,\ldots,\alpha_q)$, where $\alpha_i\in \Hom(\cO,\bbc) = \Hom(\cO,\Lambda_i)\otimes\bbr$
and $w = (b_1,\ldots,b_{p-q},\beta_1,\ldots,\beta_q)$, then
$$H(v,w) = \sum_{i=1}^{p-q} \frac{\bar a_i b_i}{\Im \tau} + \sum_{j=1}^q \frac{\bar \alpha_i(1) \beta_i(1) + \bar \alpha_i(\tau) \beta_i(\tau)}{\Im \tau_i}.$$
As $\Im \tau,\Im \tau_i>0$, this is positive definite.  For all $\sigma \in\bbc\setminus \bbr$ and $a,b,c,d\in \bbr$, we have
$$\Im \frac{(\overline{a+b\sigma}) (c+d\sigma)}{\Im \sigma} = ad-bc,$$
so $\Im H$ restricts to a perfect $\bbz$-valued anti-symmetric pairing on $\Lambda$.  It follows that $V/\Lambda$ is isomorphic to the analytic space of a principally polarized  abelian variety $A_{\Lambda_1,\ldots,\Lambda_q}$.

We define a closed subgroup scheme $G/\bbz$ of the symplectic group $\Sp_{2p+2q}/\bbz$ as follows:
For any commutative ring $R$, let $G(R)$ denote the group of $R$-linear and $\cO$-linear automorphisms of $\Lambda\otimes R$
which respect to the anti-symmetric form 
$$\Im H \colon \Lambda\otimes R\times \Lambda\otimes R\to R.$$
In particular, $G(\bbr)$ is the group of automorphisms of the real vector space $V$ which are symplectic with respect to $\Im H$ and $\cO$-linear.
If $H'$ is the Hermitian form on $V$
determined by $\Im H$ and the $\iota$-structure, then $H'$ has signature $(p,q)$, and $G(\bbr)$ can be regarded as the unitary group $U(p,q)$ defined by $H'$ and the $\iota$-structure.
For all $g\in G$, the complex torus $V/g(\Lambda)$ is polarized by the Hermitian form $g(H)$ defined by 
$$g(H)(v,w) := H(g^{-1}(v), g^{-1}(w)),$$
so $V/g(\Lambda)$ is isomorphic to the 
analytic space of a principally polarized abelian variety $A_g$ endowed with the endomorphism structure $\iota_g\colon \cO\to \End A_g$.

Let $K$ denote the subgroup of $G(\bbr)$ consisting of elements which are $\bbc$-linear with respect to the usual complex structure on $V$.  Then $K \cong U(p)\times U(q)$, and $G(\bbr)/K$ is a Hermitian symmetric domain.  By the Baily-Borel theorem, the quotient $G(\bbz)\backslash G(\bbr)/K$ is isomorphic to the analytic space associated to a complex quasi-projective variety $S$.  The variety $S$ parametrizes a family of abelian varieties $\{A_s\mid s\in S(\bbc)\}$ of dimension $p+q$ together with a principal polarization on each $A_s$ and a map $\iota\colon \cO\to \End A_s$.
This family comes from a scheme $A\to S$.

\begin{thm}
For a very general member $A_s$ of the above family of abelian varieties, there are no rational curves on the Kummer variety of $A_s$.
\end{thm}

\begin{proof}

By the weak approximation theorem, $G(\bbq)$ is dense in $G(\bbr)$ in the real topology.  It follows that the image of $G(\bbq)$ is dense in the complex topology in $S(\bbc)$
and therefore Zariski-dense in $S(\bbc)$.  By construction, all elements of $G(\bbq)$ correspond to abelian varieties isogenous to $E_1\times \cdots\times E_q\times E^{p-q}$.
Any such abelian variety admits a surjective homomorphism of abelian varieties to $B := E_1\times \cdots \times E_q$.  We have assumed that the $E_i$ are chosen to be mutually
non-isogenous and so that $\bar B$ has no diagonal rational curves.  By Lemma~\ref{diagonal vs nondegenerate}, this implies that $\bar B$ has no nondegenerate rational curves, so
by Lemma~\ref{surjection}, the Kummer variety of any abelian variety corresponding to an element of $G(\bbq)$ has no nondegenerate rational curve.

By Proposition~\ref{family}, it follows that for a very general point $s\in S(\bbc)$, the Kummer variety of $A_s$ has no rational curve.

\end{proof}

\end{document}